\documentclass[12 pt, reqno]{amsart}

\usepackage[utf8]{inputenc}
\usepackage[T2A]{fontenc}
\usepackage[english]{babel}
\usepackage[margin=2.5cm]{geometry}

\usepackage{enumitem}

\usepackage{amssymb}
\usepackage{amscd}
\usepackage{graphicx}
\usepackage[usenames,dvipsnames,svgnames]{xcolor}
\usepackage[all]{xy}
\usepackage{xspace}

\usepackage{hyperref}
\hypersetup{hidelinks=true, hypertexnames=false}
\hypersetup{
    hypertexnames=false,
    colorlinks,
    linkcolor={red!30!black},
    citecolor={blue!50!black},
    urlcolor={blue!80!black}
}

\newtheorem{theorem}[subsection]{Theorem}
\newtheorem{lemma}[subsection]{Lemma}
\newtheorem{sublemma}[subsubsection]{Lemma}

\newtheorem{corollary}[subsection]{Corollary}

\newtheorem{definition}[subsection]{Definition}

\newtheorem{remark}[subsection]{Remark}

\makeatletter
\@addtoreset{subsection}{section}
\@addtoreset{equation}{section}
\@addtoreset{figure}{section}
\@addtoreset{table}{section}
\makeatother

\makeatletter
\newcommand\testshape{family=\f@family; series=\f@series; shape=\f@shape.}
\def\myemphInternal#1{\if n\f@shape%
\begingroup\itshape #1\endgroup\/%
\else\begingroup\it\sffamily #1\endgroup%
\fi}
\def\myemph{\futurelet\testchar\MaybeOptArgmyemph}
\def\MaybeOptArgmyemph{\ifx[\testchar \let\next\OptArgmyemph
                 \else \let\next\NoOptArgmyemph \fi \next}
\def\OptArgmyemph[#1]#2{\index{#1}\myemphInternal{#2}}
\def\NoOptArgmyemph#1{\myemphInternal{#1}}
\makeatother

\newcommand\AEM[1]{#1}
\newcommand\BEM[1]{#1}

\newcommand{\bC}{\mathbb{C}}
\newcommand{\bR}{\mathbb{R}}
\newcommand{\bN}{\mathbb{N}}

\newcommand{\bZ}{\mathbb{Z}}
\newcommand{\cO}{\mathcal{O}}

\newcommand{\Int}{\mathrm{Int}}
\newcommand\ev[1]{\mathrm{ev}_{#1}}
\newcommand{\Cl}[1]{\mathop{\overline{#1}}\nolimits}

\newcommand\Asp{\BEM{A}}
\newcommand\Ksp{K}
\newcommand\Lsp{L}
\newcommand\Msp{M}

\newcommand\Xsp{X}
\newcommand\Ysp{Y} 
\newcommand\Zsp{Z}
\newcommand\Usp{U}
\newcommand\Vsp{V}
\newcommand\Wsp{W}

\newcommand\dif{h}
\newcommand\kdif{k}
\newcommand{\Partition}{\Delta}

\newcommand{\leaf}{\omega}

\newcommand{\id}{\mathop{\mathrm{id}{}}\nolimits}

\newcommand{\prj}{\AEM{p}} 

\renewcommand{\emptyset}{\varnothing}

\newcommand\End{\mathcal{E}}
\newcommand\Homeo{\mathcal{H}}

\newcommand\Hm[1]{\Homeo(#1)}
\newcommand\HXP{\Homeo(\Xsp,\Partition)}

\newcommand\EY{\End(\Ysp)}

\newcommand\HY{\Hm{\Ysp}}

\newcommand\ahom{\psi}
\newcommand\specPoints{\mathrm{Br}(\Ysp)} 

\newcommand\hcl[1]{\mathrm{hcl}(#1)}
\newcommand\hclA[2]{\mathrm{hcl}_{#2}(#1)}

\newcommand\branch{branch}
\newcommand\CONT{\ensuremath{(C)}\xspace}
\newcommand\COMP{\ensuremath{(K)}\xspace}

\newcommand{\tG}{Q}

\author{Sergiy Maksymenko}
\email{maks@imath.kiev.ua}
\address{Institute of Mathematics of NAS of Ukraine, Tereshchenkivska str. 3, Kyiv, 01024, Ukraine}

\author{Eugene Polulyakh}
\email{polulyah@imath.kiev.ua}
\address{Institute of Mathematics of NAS of Ukraine, Tereshchenkivska str. 3, Kyiv, 01024, Ukraine}

\title[Actions of foliated homeomorphisms on spaces of leaves]{Actions of groups of foliated homeomorphisms on spaces of leaves}

\begin{document}

\begin{abstract}
Let $\Delta$ be a foliation on a topological manifold $X$, $Y$ be the space of leaves, and $p:X\to Y$ be the natural projection.
Endow $Y$ with the factor topology with respect to $p$.
Then the group $\mathcal{H}(X,\Delta)$ of foliated (i.e. mapping leaves onto leaves) homeomorphisms of $X$ naturally acts on the space of leaves $Y$, which gives a homomorphism $\psi:\mathcal{H}(X,\Delta) \to \mathcal{H}(Y)$.
We present sufficient conditions when $\ahom$ is continuous with respect to the corresponding compact open topologies.

In fact similar results hold not only for foliations but for a more general class of partitions $\Delta$ of locally compact Hausdorff spaces $X$.
\end{abstract}
\maketitle

\section{Introduction}
Let $\Xsp$ be an $m$-dimensional topological manifold and $\Partition$ be a foliation on $\Xsp$.
Denote by $\Ysp$ the space of leaves of $\Partition$ and let $\prj:\Xsp\to\Ysp$ be the natural projection associating to each $x\in\Xsp$ the leaf containing $x$.
Endow $\Ysp$ with the \myemph{factor} topology with respect to $\prj$, so a subset $A\subset\Ysp$ is open if and only if $\prj^{-1}(A)$ is open in $\Xsp$.

A homeomorphism $\dif:\Xsp\to\Xsp$ will be called
\begin{itemize}
\item
\myemph{foliated} if for each leaf $y$ of $\Partition$ its image $\dif(y)$ is also a leaf of $\Partition$;
\item
\myemph{leaf-preserving} if $\dif(y)=y$ for each leaf $y$ of $\Partition$.
\end{itemize}

Obviously, each leaf-preserving homeomorphism is foliated.

Denote by $\HXP$ the group of all foliated homeomorphisms of $\Xsp$ and by $\HY$ the group of all homeomorphisms of $\Ysp$.
Endow these groups with the compact open topologies.
Notice that in general the multiplication and inversion are not continuous operations in $\HXP$ and $\HY$ with respect to compact open topologies.

The spaces of leaves $\Ysp$ of foliations often appear as spaces of orbits of flows and more generally of group actions and play an important role in the understanding the dynamics of that actions, e.g.~\cite{BouacidaEchiSalhi:JMSJ:2000, HattabSalhi:TP:2004,  BonattiHattabSalhiVago:TA:2011, Echi:TA:2012, Gauld:NonMetrManif:2014, BaillifGabardGauld:ProcAMS:2014, BarthelmeGogolev:MRL:2019} and others.
The usual difficulty arising at once when we pass from the manifold $\Xsp$ to the space of leaves $\Ysp$ is that $\Ysp$ is usually non-Hausdorff.
Moreover, if some leaves of $\Partition$ are non-closed as subsets of $\Xsp$, (e.g. dense in some open subset), then $\Ysp$ is not $T_1$ as well.
Let us also mention that in \cite{Gauld_vanMill:NZJM:2012} it was given a characterization of a manifold $\Ysp$ to be metrizable in terms of $\HY$ endowed with compact open topology.

%

Homotopy properties of groups $\HXP$ were studied e.g. in~\cite{Rybicki:MonM:1995, Rybicki:SJM:1996, Rybicki:DM:1996, Rybicki:DGA:1999, Rybicki:DGA:2001, HallerTeichmann:AGAG:2003, AbeFukui:CEJM:2005, LechRybicki:BCP:2007, Maksymenko:OsakaJM:2011} and references therein.
Most of them extend the results by M.~Herman~\cite{Herman:CR:1971}, W.~Thurston~\cite{Thurston:BAMS:1974}, J.~Mather~\cite{Mather:Top:1971, Mather:BAMS:1974} and D.~B.~A.~Epstein~\cite{Epstein:CompMath:1970} about perfectness of such groups (i.e. triviality homologies of their classifying spaces).

The aim of the present paper is to propose a certain approach for relating the homotopy types of the groups $\HXP$ and $\HY$.
Notice that each $\dif\in\HXP$ yields a permutation $\ahom(\dif):\Ysp\to\Ysp$ of the leaves of $\Partition$ such that the following diagram is commutative:
\begin{equation}\label{equ:psi_diagram}
\begin{CD}
\Xsp @>{\dif}>> \Xsp \\
@V{\prj}VV  @VV{\prj}V \\
\Ysp @>{\ahom(\dif)}>> \Ysp
\end{CD}
\end{equation}
One can easily check that $\ahom(\dif)$ is in fact a \myemph{homeomorphism} of $\Ysp$, see e.g. Lemma~\ref{lm:f_pg} below, while the correspondence $\dif\to\ahom(\dif)$ is a well-defined \myemph{homomorphism} $\ahom:\HXP\to\HY$.
Evidently, its kernel $\ker(\ahom)$ consists of \myemph{leaf-preserving} homeomorphisms.

We will study the question whether the following natural property holds:

\begin{enumerate}[label=\CONT]
\item\label{pr:ahom_cont}
The homomorphism $\ahom:\HXP\to\HY$ is continuous with respect to the corresponding compact open topologies of those groups.
\end{enumerate}

It is satisfied in many special cases, however the authors were not able to find its general investigations in the available literature.
One of the reasons is that the space of leaves is usually non-Hausdorff, while compact open topologies are well-studied for locally compact Hausdorff spaces, e.g.~\cite{Fox:BAMS:1945}.


Under essentially more general settings than foliations we will show, see Lemma~\ref{lm:cond_for_cont_ahom}, that property~\ref{pr:ahom_cont}
is a consequence of
the following condition:
\begin{enumerate}[label=\COMP]
\item\label{pr:ahomL_K}
For every compact $\Lsp\subset\Ysp$ there exists a compact $\Ksp\subset\Xsp$ such that $\prj(\Ksp)=\Lsp$.
\end{enumerate}

Further we will describe several particular situations when~\ref{pr:ahomL_K} and therefore~\ref{pr:ahom_cont} hold, see Corollary~\ref{cor:cond_for_cont_ahom}.
On the other hand, we will also present examples when \CONT holds, while \COMP fails, see examples in  Seciton~\ref{sect:examples}.
They are well-known partition of torus by orbits of irrational flow and partition into orbits of Denjoy homeomorphism of the circle.

The following Theorem~\ref{th:fK_L__foliations} is one of the principal results of the paper.

Say that a subset $\Asp\subset\Ysp$ is \myemph{locally finite}, if every $y\in \Ysp$ has a neighborhood $\Usp$ such that $\Usp\cap\Asp$ is a finite set.

We also say that points $y,z\in\Ysp$ are \myemph{$T_2$-disjoint}, if they have disjoint neighborhoods.
If a point $y\in\Ysp$ is not $T_2$-disjoint from some other point $z\in\Ysp$, then $y$ will be called a \myemph{\branch} point of $\Ysp$, see Section~\ref{sect:branch_points} for details.

\begin{theorem}\label{th:fK_L__foliations}
Let $\Partition$ be a foliation on a Hausdorff topological manifold $\Xsp$.
Suppose that
\begin{enumerate}[label={\rm\alph*)}]
\item the space of leaves $\Ysp$ of $\Partition$ is a $T_1$-space, i.e. each leaf $\leaf$ of $\Partition$ is a closed subset of $\Xsp$;
\item the set of \branch\ points of $\Ysp$ is locally finite.
\end{enumerate}
Then the map~\eqref{equ:ahom_EXP_EY} $\ahom:\HXP\to\HY$ is continuous with respect to compact open topologies.
\end{theorem}

This theorem is a special case of Theorem~\ref{th:fK_L} and will be proved in subsection~\ref{sect:ect:proof:th:fK_L__foliations}.


\subsection*{Evaluation maps}
Property~\ref{pr:ahom_cont} has also several consequences which relate the homotopy groups of $\HXP$ and $\HY$.

Let $x\in\Xsp$, $y = \prj(x) \in\Ysp$ be the leaf containing $x$ and
\begin{align*}
	&\ev{x}:\HXP \to \Xsp,   && \ev{y}:\HY \to \Ysp, \\
	&\ev{x}(\dif) = \dif(x),  && \ev{y}(g)= g(y),
\end{align*}
be the \myemph{evaluation} maps at $x$ and $y$ respectively.
It is well known and is easy to show that $\ev{x}$ and $\ev{y}$ are continuous, see Lemma~\ref{lm:comp_open_top}.

Then commutativity of the diagram~\eqref{equ:psi_diagram} implies another commutative diagram:
\begin{equation}\label{equ:ev_diagram}
\begin{CD}
\HXP @>{\ahom}>> \HY \\
@V{\ev{x}}VV  @VV{\ev{y}}V \\
\Xsp @>{\prj}>> \Ysp
\end{CD}
\end{equation}
Indeed,
\begin{align*}
	\prj \circ \ev{x}(\dif) = \prj \bigl(\dif(x)\bigr) =
	\ahom(\dif)  \bigl(\prj(x)\bigr)  =
		\ahom(\dif) (y)  =  \ev{y}\circ \ahom(\dif).
\end{align*}

Till the end of this section assume that property~\ref{pr:ahom_cont} holds, that is $\ahom$ is continuous, and so~\eqref{equ:ev_diagram} consists of continuous maps.
Then the following statements hold.

1) For each $n\geq0$ diagram~\eqref{equ:ev_diagram} induces a diagram consisting of the corresponding $k$-th homotopy groups and induced homomorphisms:
\begin{equation}\label{equ:ev_diagram_pik}
\begin{CD}
\pi_k\bigl( \HXP, \id_{\Xsp} \bigr) @>{\ahom_k}>> \pi_k\bigl(\HY, \id_{\Ysp}\bigr) \\
@V{(\ev{x})_k}VV  @VV{(\ev{y})_k}V \\
\pi_k(\Xsp,x) @>{\prj_k}>> \pi_k(\Ysp,y)
\end{CD}
\end{equation}
Let us mention that for $k=0$ the sets $\pi_0\bigl( \HXP, \id_{\Xsp} \bigr)$ and $\pi_0\bigl(\HY, \id_{\Ysp}\bigr)$ are in fact groups, and the induced map
\[
\ahom_0: \pi_0\bigl( \HXP, \id_{\Xsp} \bigr) \to \pi_0\bigl(\HY, \id_{\Ysp}\bigr)
\]
of those groups is a \myemph{homomorphism} as well, see Lemma~\ref{lm:top_monoid}.

2) There is a similar diagram for any subgroup $\mathcal{G}$ of $\HXP$:
\begin{equation}\label{equ:ev_diagram_G}
\begin{CD}
\mathcal{G} @>{\ahom}>> \HY \\
@V{\ev{x}}VV  @VV{\ev{y}}V \\
\Xsp @>{\prj}>> \Ysp
\end{CD}
\end{equation}

3) Finally, let us consider a very important case when $\Xsp$ is a \myemph{smooth} manifold, $\Partition$ is a \myemph{smooth} foliation, and $\mathcal{G}$ is the group of foliated diffeomorphisms of $\Xsp$.
Then $\mathcal{G}$ is usually endowed with a \myemph{strong} or \myemph{weak} $C^{r}$ topology for some $r\geq0$.
Since the \myemph{compact open topology} is the same as \myemph{weak $C^{0}$} and is the weakest among all the above topologies, we see that the homomorphism $\ahom:\mathcal{G}\to\HY$ is still continuous into compact open topology of $\HY$, and we still get the commutative diagram:
\begin{equation}\label{equ:ev_diagram_G_pik}
\begin{CD}
\pi_k\bigl( \mathcal{G}, \id_{\Xsp} \bigr) @>{\ahom}>> \pi_k\bigl(\HY, \id_{\Ysp}\bigr) \\
@V{\ev{x}}VV  @VV{\ev{y}}V \\
\pi_k(\Xsp,x) @>{\prj}>> \pi_k(\Ysp,y)
\end{CD}
\end{equation}


\subsection*{Structure of the paper}

In the preliminary section~\ref{sect:preliminaries} we discuss topological monoids and non-Hausdorff locally compact spaces.
It is shown that many properties of compact open topologies on the spaces of continuous maps between locally compact Hausdorff topological spaces are preserved if we omit Hausdorff property.
In fact, the presented results are known and in some cases are rather simple, but we give their proofs just to assure that we do not assume Hausdorff property.

In Section~\ref{sect:arb_partitions} we also consider arbitrary partitions not only foliations, and prove in Lemma~\ref{lm:cond_for_cont_ahom} that condition \COMP implies \CONT.

Section~\ref{sect:branch_points} is devoted to the study of so called \branch\ points at which the space of leaves looses Hausdorff property.
Finally in section~\ref{sect:loc_fin_subsets} we prove Theorem~\ref{th:fK_L} and its particular case Theorem~\ref{th:fK_L__foliations}.

\section{Preliminaries}\label{sect:preliminaries}
%
%
%

\subsection*{Topological monoids}
Recall that a \myemph{monoid} structure on a set $G$ is a map $\mu:G \times G\to G$ being associative and having a \myemph{unit} element $e\in G$, that is $\mu(a,\mu(b,c))=\mu(\mu(a,b),c)$ and $\mu(a,e) = \mu(e,a)=a$ for all $a,b,c\in G$.
We will usually denote $\mu(a,b)$ simply by $ab$.

Notice that the unit element in a monoid is unique.
Indeed, if $e'$ is another unit, then $e' = ee' = e$.

A \myemph{homomorphism} of monoids $q:G \to H$ with units $e_G$ and $e_H$ respectively is a map such that  $q(e_G) = e_H$ and $q(ab) = q(a) q(b)$ for all $a,b\in G$.

If a monoid $G$ is endowed with a topology in which $\mu$ is continuous, then it is called a \myemph{topological monoid}.


Let $G$ be a topological monoid.
For each $a\in G$ denote by $G_a$ the path component of $G$ containing $a$.
Recall that the set of all path components of $G$ with a distinguished element $G_e$ is denoted by $\pi_0(G,e)$ and called $0$-th homotopy set of $G$ at $e$.
Formally it can be defined as the set of homotopy classes
\[
\pi_0(G,e) = \bigl[ ( \{0,1\}, 0 ), (G, e) \bigr]
\]
of maps of $0$-dimensional sphere $S^{0}\equiv\partial[0,1]\equiv\{0,1\} \to G$ sending $0$ to $e$.
In general $0$-th homotopy set of a topological space is not a group in contrast to other sets $\pi_k(G,e)$ with $k\geq1$.
However, as the following easy and well known lemma claims, $\pi_0(G,e)$ inherits algebraic structure of $G$, therefore it is a topological monoid or a topological group.

\begin{lemma}\label{lm:top_monoid}
\begin{enumerate}[label={\rm(\alph*)}, wide, itemsep=1ex]
\item\label{enum:top_monoid:homo}
Let $q:G \to H$ be a homomorphism of monoids with units $e_G$ and $e_H$ respectively.
If $a\in G$ is invertible, then $q(a)^{-1} = q(a^{-1})$ is invertible in $H$.
In particular, if $G$ is a group, and $q$ is surjective, then $H$ is a group as well.

\item\label{enum:top_monoid:Ge}
Let $G$ be a topological monoid.
Then $G_a G_b:= \mu(G_a\times G_b) \subset G_{ab}$ for all $a,b\in G$.
Hence one can define a \myemph{monoid} structure on $\pi_0(G,e)$ by $G_a * G_b = G_{ab}$, so that the natural projection $q:G \to \pi_0(G,e)$ defined by $q(a) = G_a$ becomes a morphism of monoids.

%

If $G$ is a group (not necessary topological, i.e.\! the inversion map is not necessarily continuous), then $G_e$ is a normal subgroup, $G_a = a G_e = G_e a$ for each $a\in G$, so the above map $q$ is a composition $q:G \to G/G_e \cong \pi_0(G,e)$.



\item\label{enum:top_monoid:morphism_of_monoids}
Let $\ahom:G \to H$ be a continuous homomorphism of topological monoids with unit elements $e_G$ and $e_H$ respectively.
Then the induced mapping $\ahom_0:\pi_0(G,e_G) \to \pi_0(H,e_H)$ is a homomorphism of monoids.
If $\pi_0(G,e_G)$ and $\pi_0(H,e_H)$ are groups, then $\ahom_0$ is a homomorphism of groups.
\end{enumerate}
\end{lemma}
\begin{proof}
\ref{enum:top_monoid:Ge}
Since by definition the sets $G_a$ and $G_b$ are path connected, their product $G_a G_b := \mu(G_a\times G_b)$ is path connected as well.
Moreover, as it contains $ab$, it follows that $\mu(G_a\times G_b) \subset G_{ab}$.

All other statements of lemma are easy we leave them for the reader.
\end{proof}
%
%
%
%
%
%
%
%
%

\subsection*{Locally compact spaces}
Let $\Xsp$ be a topological space.
\begin{definition}\label{def:loc_comp}
We will say that $\Xsp$ is \myemph{locally compact} if every point $x\in\Xsp$ has a local base consisting of compact neighborhoods.
In other words, for every open $U$ containing $x$, there exists a compact set $K$ such that $x\in \Int{K} \subset K \subset U$.
\end{definition}

\begin{remark}
There are some other approaches for defining local compactness.
For example:
\begin{enumerate}[label={\rm(\arabic*)}]
\item\label{enum:loc_comp:comp_nbh}
every point $x\in\Xsp$ has a compact neighborhood, i.e.\! there is a compact set $K$ such that $x\subset \Int{K}$;
\item\label{enum:loc_comp:closed_comp_nbh}
every point $x\in\Xsp$ has a \myemph{closed} compact neighborhood.
\end{enumerate}
It is well known and is easy to check that~\ref{enum:loc_comp:closed_comp_nbh} is not equivalent to Definition~\ref{def:loc_comp}, and each of these definitions implies~\ref{enum:loc_comp:comp_nbh}.
Moreover, all these definitions coincide for Hausdorff spaces.
In the present paper a local compactness is always used in the sense of Definition~\ref{def:loc_comp}.
\end{remark}

\begin{lemma}\label{lm:loc_comp_prop}
Let $\Xsp$ be a locally compact topological space in the sense of Definition~\ref{def:loc_comp}.
\begin{enumerate}[label={\rm(\arabic*)}, leftmargin=*]
\item\label{enum:loc_comp:nbh_of_compact}
For every compact $\Ksp$ and open $\Usp$ in $\Xsp$ with $\Ksp \subset \Usp$ there exists a compact $\Lsp$ such that $\Ksp \subset \Int{\Lsp} \subset\Lsp \subset \Usp$.

\item\label{enum:loc_comp:pres_under_open}
Let $\prj:\Xsp\to\Ysp$ be an open surjective map.
Then $\Ysp$ is also locally compact.
\end{enumerate}
\end{lemma}
\begin{proof}
\ref{enum:loc_comp:nbh_of_compact}
For each $x\in\Ksp$ there exists a compact subset $\Lsp_x$ such that $x\in \Int{\Lsp_x} \subset \Lsp_x \subset \Usp$.
Due to compactness of $\Ksp$ one can find finitely many points $x_1,\ldots,x_n\in\Ksp$ such that $\Ksp\subset\mathop{\cup}\limits_{i=1}^{n} \Int{\Lsp_{x_i}}$.
Put $\Lsp = \mathop{\cup}\limits_{i=1}^{n} \Lsp_{x_i}$.
Then $\Lsp$ is compact and $\Ksp \subset \Int{\Lsp} \subset\Lsp \subset \Usp$.

\ref{enum:loc_comp:pres_under_open}
Let $y\in\Ysp$ and $\Vsp$ be an open neighborhood of $y$.
We should find a compact subset $\Lsp \subset \Ysp$ with $y\in\Int{\Lsp} \subset\Lsp \subset \Vsp$.

Fix a point $x\in \Xsp$ with $\prj(x)=y$, and let $\Usp =\prj^{-1}(\Vsp)$.
Since $\Xsp$ is locally compact there exists a compact $K\subset \Xsp$ such that $x \in \Int{K} \subset K\subset \Usp$.
Then $\Lsp = \prj(\Ksp)$ is compact.
Moreover, as $\prj$ is open, $\prj(\Int{K})$ is an open neighborhood of $y$ contained in $\Lsp$, and so $y\in \prj(\Int{\Ksp}) \subset \Int{\Lsp} \subset \Lsp \subset \Vsp$.
Thus $\Ysp$ is locally compact.
\end{proof}
%

\newcommand\CXY{C(\Xsp,\Ysp)}
\newcommand\CXZ{C(\Xsp,\Zsp)}
\newcommand\CYZ{C(\Ysp,\Zsp)}
\newcommand\NKU[2]{\mathcal{N}(#1,#2)}
\newcommand\NKUind[3]{\mathcal{N}_{#1}(#2,#3)}
\newcommand\CYY{C(\Ysp,\Ysp)}

\subsection*{Compact open topologies}
Let $\Xsp, \Ysp$ be topological spaces and $\CXY$ be the set of all continuous maps $f:\Xsp\to\Ysp$.
For every compact $K\subset \Xsp$ and open $U\subset \Ysp$ put
\[
\NKU{\Ksp}{\Usp} = \{ f\in\CXY \mid f(\Ksp) \subset \Usp\}.
\]
Then the \myemph{compact open} topology on $\CXY$ is the topology generated by the prebase consisting of sets $\NKU{\Ksp}{\Usp}$, where $\Ksp$ runs over all compact subsets of $\Xsp$ and $\Usp$ runs over all open subsets of $\Ysp$.

For a point $y\in\Ysp$ we will denote by $\ev{y}:\CXY\to\Ysp$ the ``evaluation map at $y$'' defined by $\ev{y}(f) = f(y)$.

\begin{lemma}\label{lm:comp_open_top}
Let $\Xsp,\Ysp,\Zsp$ be topological spaces.
\begin{enumerate}[label={\rm(\arabic*)}, leftmargin=*, itemsep=1ex]
\item\label{enum:comp_open_top:ev}
For each $y\in\Ysp$ the evaluation map $\ev{y}:\CXY\to\Ysp$ is continuous.

\item\label{enum:comp_open_top:composition}
Let $\mu: \CXY\times \CYZ \to \CXZ$ be the composition map defined by $\mu(f,g) = g\circ f$.
If $\Ysp$ is locally compact, then $\mu$ is continuous with respect to compact open topologies on those spaces.

\item\label{enum:comp_open_top:top_monoid}
Suppose again that $\Ysp$ is locally compact.
Then $\CYY$ is a topological monoid, that is $\mu$ is continuous with respect to compact open topologies on those spaces.
Hence, due to Lemma~\ref{lm:top_monoid}, $\pi_0\CYY$ is a monoid, while $\pi_0\HY$ is a group.
\end{enumerate}
\end{lemma}
\begin{proof}
\ref{enum:comp_open_top:ev}
Let $\Vsp \subset \Ysp$ be an open subset.
Then $\ev{y}^{-1}(\Vsp) = \NKU{\{y\}}{\Vsp}$.
Indeed, $g\in\CXY$ belongs to $\NKU{\{y\}}{\Vsp}$ iff $\ev{y}(g) = g(y)\in\Vsp$.
Hence the inverse image of each open set in $\Ysp$ is open in $\CXY$ and so $\ev{y}$ is continuous.

\ref{enum:comp_open_top:composition}
Let $f\in\CXY$, $g\in\CYZ$, and $\NKU{\Ksp}{\Wsp}$ be an open prebase neighborhood of $g\circ f$ in $\CXZ$, where $\Ksp \subset \Xsp$ is a compact and $\Wsp\subset\Zsp$ is open.
Denote $\Lsp = f(\Ksp)$ and $\Vsp = g^{-1}(\Wsp)$.
Then $\Lsp$ is compact, $\Vsp$ is open and $\Lsp \subset \Vsp$.
Hence by~\ref{enum:loc_comp:nbh_of_compact} of Lemma~\ref{lm:loc_comp_prop} there exists a compact set $\Msp \subset \Ysp$ such that $\Lsp\subset  \Int{\Msp} \subset \Msp\subset \Vsp$.

Then $\NKU{\Ksp}{\Int{\Msp}}$ is an open neighborhood of $f$ in $\CXY$, and $\NKU{\Lsp}{\Wsp}$ is an open neighborhood of $g$ in $\CYZ$.
Moreover, we claim that
\[
\mu\bigl(\,\NKU{\Ksp}{\Int{\Msp}} \times \NKU{\Msp}{\Wsp}\,\bigr) \ \subset \ \NKU{\Ksp}{\Wsp}.
\]
Indeed, if $f'\in \NKU{\Ksp}{\Int{\Msp}}$ and $g'\in\NKU{\Msp}{\Wsp}$, then
\[  g'\circ f'(\Ksp) \subset g'(\Int{\Msp}) \subset g'(\Msp) \subset \Wsp,\]
that is $\mu(f', g') = g'\circ f' \in \NKU{\Ksp}{\Wsp}$.

\medskip

Statement~\ref{enum:comp_open_top:top_monoid} is a direct consequence of~\ref{enum:comp_open_top:composition} when $\Xsp=\Ysp=\Zsp$ and Lemma~\ref{lm:top_monoid}.
\end{proof}

\section{Maps consistent with a partition}\label{sect:arb_partitions}
Let $\prj:\Xsp\to\Ysp$ be a \myemph{factor} map between topological spaces, that is $\prj$ is surjective and a subset $A\subset \Ysp$ is open if and only if $\prj^{-1}(A)$ is open in $\Xsp$.
In other words, $\Ysp$ has the strongest topology in which $\prj$ is continuous.
It is well known and is easy to see that every \myemph{open} and every \myemph{closed} map is factor.

Let $\Partition = \{\prj^{-1}(y) \mid y\in \Ysp \}$ be the partition of $\Xsp$ into the inverse images of points of $\Ysp$.
A continuous map $\dif:\Xsp\to\Xsp$ will be called a \myemph{$\Partition$-map} if for each $\leaf\in\Partition$ its image $\dif(\leaf)$ is contained in some element $\leaf'$ of $\Partition$.
Hence every $\Partition$-map $\dif$ induces a map $\ahom(\dif):\Ysp\to\Ysp$ making commutative the following diagram:
\begin{equation}\label{equ:psi_diagram1}
\begin{CD}
\Xsp @>{\dif}>> \Xsp \\
@V{\prj}VV  @VV{\prj}V \\
\Ysp @>{\ahom(\dif)}>> \Ysp
\end{CD}
\end{equation}

The following well known and easy lemma implies that $\ahom(\dif)$ is continuous whenever $\dif$ is so.

\begin{lemma}\label{lm:f_pg}
Suppose we have the following commutative diagram
\begin{equation}\label{equ:fpg_diagram}
\xymatrix{
	\Xsp \ar[rd]^-{f} \ar[d]_-{\prj} \\
	\Ysp \ar[r]_-{g} & Z
}
\end{equation}
in which $f$ is continuous and $\prj$ is a factor map.
Then $g$ is continuous as well.
\end{lemma}
\begin{proof}
We should show that for each open $A\subset Z$ its inverse $g^{-1}(A)$ is open in $\Ysp$ as well.
As $\prj$ is a factor map, the latter is equivalent to the assumption that $\prj^{-1}(g^{-1}(A))$ is open in $\Xsp$.
But $\prj^{-1}(g^{-1}(A)) = f^{-1}(A)$ due to commutativity of the diagram~\eqref{equ:fpg_diagram} and this set is open as $f$ is continuous.
\end{proof}

Let $\End(\Xsp, \Partition)$ be the monoid of all $\Partition$-maps of $\Xsp$, and $\End(\Ysp)=C(\Ysp,\Ysp)$ be the monoid of all continuous self-maps of $\Ysp$.
Let also $\Homeo(\Xsp, \Partition)$ be the subgroup of $\End(\Xsp, \Partition)$ consisting of homeomorphisms and $\HY$ be the group of homeomorphisms of $\Ysp$.

Then Lemma~\ref{lm:f_pg} implies that the correspondence $\dif\mapsto\ahom(\dif)$ is a well defined map
\begin{equation}\label{equ:ahom_EXP_EY}
\ahom:\End(\Xsp, \Partition) \to \End(\Ysp)
\end{equation}
being a homomorphism of monoids.

\begin{definition}\label{def:CK}
We say that the quotient map $\prj:\Xsp\to\Ysp$ has
\begin{itemize}[leftmargin=*, itemsep=1ex]
\item
property \COMP if for every compact $\Lsp\subset\Ysp$ there exists a compact $\Ksp\subset\Xsp$ such that $\prj(\Ksp)=\Lsp$;

\item
property \CONT whenever the homomorphism $\ahom:\End(\Xsp, \Partition) \to \End(\Ysp)$ is continuous with respect to the corresponding compact open topologies of those groups.
\end{itemize}
\end{definition}

The following statement gives sufficient conditions under which $\ahom$ will be continuous with respect to compact open topologies on $\End(\Xsp, \Partition)$ and  $\End(\Ysp)$.

\begin{lemma}\label{lm:cond_for_cont_ahom}
 \COMP $\Longrightarrow$ \CONT.
\end{lemma}
\begin{proof}
Recall that the prebase of compact open topology on $\EY$ consists of sets
\[ \NKU{\Lsp}{\Vsp} = \{ \kdif\in\EY \mid \kdif(\Lsp) \subset \Vsp \},\] where $\Lsp\subset \Ysp$ is compact and $\Vsp \subset \Ysp$ is open.
Let $\dif \in \End(\Xsp, \Partition)$ and $\NKU{\Lsp}{\Vsp}$ be any prebase neighborhood of $\ahom(\dif)$ in $\EY$, so $\ahom(\dif)(\Lsp)\subset \Vsp$.
Fix any compact $\Ksp\subset \Xsp$ with $\prj(\Ksp)=\Lsp$ and put $\Usp = \prj^{-1}(\Vsp)$.
Then
\[
\dif(\Ksp)  \subset \dif(\prj^{-1}(\Lsp)) \subset \prj^{-1}(\ahom(\dif)(\Lsp)) \subset \prj^{-1}(\Vsp) = \Usp,
\]
so $\dif\in \NKU{\Ksp}{\Usp}$.
Moreover, we claim that in fact $\ahom(\NKU{\Ksp}{\Usp}) \subset \NKU{\Lsp}{\Vsp}$, which will imply continuity of $\ahom$ at $\dif$.

Indeed, let $g\in \NKU{\Ksp}{\Usp}$, so $g(\Ksp) \subset \Usp$.
Then
\[
\ahom(g)(\Lsp) =
\ahom(g)\bigl( \prj(\Ksp) \bigr)=
\prj \bigl( g(\Ksp) \bigr) \subset \prj(\Usp) = \Vsp.
\]
Thus $\ahom(g)\in \NKU{\Lsp}{\Vsp}$, and so $\ahom$ is continuous.
\end{proof}



\begin{corollary}\label{cor:cond_for_cont_ahom}
Suppose $\Ysp$ is a locally compact Hausdorff space and $\prj:\Xsp\to\Ysp$ be a surjective continuous map.
Then each of the following conditions implies condition \CONT for $\prj$:
\begin{enumerate}[label={\rm(\alph*)}, itemsep=1ex]
\item\label{enum:condstar:proper}
$\prj$ is a \myemph{proper} map, i.e. $\prj^{-1}(\Lsp)$ is compact for each compact $\Lsp\subset \Ysp$
\item\label{enum:condstar:loc_cross_sect}
$\prj$ is an open map and \myemph{admits local cross sections}, i.e. for every $y\in\Ysp$ there exists an open neighborhood $\Vsp$ and a continuous map $f:\Vsp\to\Xsp$ such that $\prj\circ f = \id_{\Vsp}$;
\item\label{enum:condstar:loc_triv_fibr}
$\prj$ is a locally trivial fibration.
\end{enumerate}
\end{corollary}
\begin{proof}
Notice that~\ref{enum:condstar:loc_triv_fibr} is a particular case of~\ref{enum:condstar:loc_cross_sect}.
In the remaining cases~\ref{enum:condstar:proper} and~\ref{enum:condstar:loc_cross_sect} it suffices to check that $\prj$ is a factor map and that condition~\COMP of Lemma~\ref{lm:cond_for_cont_ahom} holds true.

\ref{enum:condstar:proper}
It is well known and is easy to check that a proper map $\prj$ onto a locally compact Hausdorff space $\Ysp$ is closed, i.e. $\prj(F)$ is closed in $\Ysp$ for each  closed $F\subset \Xsp$.
In particular, $\prj$ is a factor map.

Moreover, suppose $\Lsp \subset \Ysp$ is a compact subset.
As $\prj$ is proper, the set $\Ksp = \prj^{-1}(\Lsp)$ is compact.
Moreover, $\prj(\Ksp)=\Lsp$ since $\prj$ is also surjective.

\ref{enum:condstar:loc_cross_sect}
Since $\prj$ is open, it is also a factor map.

Furthermore, let $\Lsp \subset \Ysp$ be a compact subset.
For each $y\in\Lsp$ choose a neighborhood $\Vsp_y$ and a cross section $f_y:\Vsp_y\to\Xsp$, i.e.\! a continuous map satisfying $\prj\circ f_y = \id_{\Vsp_y}$.
As $\Ysp$ is locally compact, one can assume that $\Vsp_y$ is compact as well.
Due to compactness of $\Lsp$ one can find finitely many points $y_1,\ldots,y_n$ such that $\Lsp \subset \mathop{\cup}\limits_{i=1}^{n} \Vsp_{y_i}$.
Moreover, as $\Ysp$ is Haussdorf, $\Lsp_i = \Lsp\cap \Vsp_{y_i}$ is compact, whence $\Ksp_i:= f_{y_i}(\Lsp_i)$ is compact in $\Xsp$.
Put $\Ksp = \mathop{\cup}\limits_{i=1}^{n} \Ksp_i$.
Then $\Ksp$ is compact and
\[
\prj(\Ksp) =  \prj\Bigl( \mathop{\cup}\limits_{i=1}^{n} \Ksp_i \Bigr) =
\mathop{\cup}\limits_{i=1}^{n} \prj(\Ksp_i) =
\mathop{\cup}\limits_{i=1}^{n} \prj \circ f(\Lsp_i) =
\mathop{\cup}\limits_{i=1}^{n} \Lsp_i = \Lsp.
\]
Corollary is proved.
\end{proof}

Notice that the proof of Corollary~\ref{cor:cond_for_cont_ahom} is heavily based on the assumption that $\Ysp$ is locally compact Hausdorff.
In the next section we will release those conditions by allowing $\Ysp$ to be ``Hausdorff'' except for some locally finite subset, see Theorem~\ref{th:fK_L}.

%
%

\section{Branch points of $T_1$ spaces}\label{sect:branch_points}

Let $\Ysp$ be a topological space.
Say that two points $y,z\in\Ysp$ are \myemph{$T_2$-disjoint} (in $\Ysp$) if they have disjoint neighborhoods.
Denote by $\hcl{y}$ the set of all $z\in \Ysp$ that are \myemph{not $T_2$-disjoint from $y$}.
Then $z\in\hcl{y}$ if and only if each neighborhood of $z$ intersects each neighborhood of $y$.
We will call $\hcl{y}$ the \myemph{Hausdorff closure} of $y$.

Evidently, $y\in\hcl{y}$.
Moreover, $y\in\hcl{z}$ if and only if $z\in\hcl{y}$.
Thus the relation $y\in\hcl{z}$ is reflexive and symmetric, however in general it is not transitive.

Following~\cite{HaefligerReeb:EM:1957} and \cite{GodbillonReeb:EM:1966} we will say that a point $y\in\Ysp$ is \myemph{\branch} whenever $\hcl{y} \setminus y \not=\emptyset$, so there are points that are not $T_2$-disjoint from $y$.
The set of all \branch\ points of $\Ysp$ will be denoted by $\specPoints$.

\begin{remark}\rm
In~\cite{MaksymenkoPolulyakh:PGC:2015, MaksymenkoPolulyakh:PGC:2016, MaksymenkoPolulyakh:MFAT:2016, MaksymenkoPolulyakhSoroka:PICG:2016, MaksymenkoPolulyakh:PGC:2017} we called those points ``\myemph{special}'', but in the present paper we decided to change their name to ``\myemph{\branch}'' as in~\cite{HaefligerReeb:EM:1957, GodbillonReeb:EM:1966} since it better reflects the structure of a space near such points.
Also in~\cite{KentMimnaJamal:IJMMS:2009} there were considered families of pairwise $T_2$-non-disjoint points called \myemph{sets of compatible appartion points}.
\end{remark}

Thus in the above notation the following conditions are equivalent:
\begin{align*}
	\, \text{(a)}~ & \text{$\Ysp$ is Hausdorff}; & \ \
	\, \text{(b)}~ & \text{$\hcl{y} =\{y\}$ for all $y\in\Ysp$;} & \ \
	\, \text{(c)}~ & \text{$\specPoints = \varnothing$.}
\end{align*}
Notice also that $\hcl{y}$ coincides with the intersection of closures of all neighborhoods of $y$:
\begin{equation}\label{equ:hcl}
\hcl{y} = \bigcap_{\text{$\Vsp$ is a neighborhood of $y$}} \overline{\Vsp}.
\end{equation}

Let now $L\subset\Ysp$ be a subset, $y\in L$, and $\hclA{y}{L}$ be the set of all points $z\in L$ that are not $T_2$-disjiont from $y$ in $L$ with respect to the topology induced from $\Ysp$.
It is straightforward that
\begin{equation}\label{equ:hclA}
\hclA{y}{L} \subseteq \hcl{y} \cap L,
\end{equation}
however the opposite inclusion can fail.

\begin{lemma}\label{lm:specpt}
For a subset $\Lsp\subset\Ysp$ the following statements hold true.
\begin{enumerate}[label={\rm(\arabic*)}, itemsep=0.5ex, leftmargin=*]
\item\label{enum:lm:specpt:Y_is_h}
If $\Lsp \cap \hcl{y} = \{y\}$ for all $y\in \Lsp$, then $\Lsp$ is Hausdorff.

\item\label{enum:lm:specpt:A_is_h}
For every $z\in \specPoints$ the following subspace $B$ of $\Ysp$ is Hausdorff:
\[
B := (\Lsp\setminus\specPoints)\cup \{z\} =
\Lsp\setminus \bigl( \specPoints \setminus \{z\} \bigr)
\]

\item\label{enum:lm:specpt:comp}
If $\Lsp$ is compact, then $\Cl{\Lsp} \subset  \Lsp \cup \specPoints$.
\end{enumerate}
\end{lemma}
\begin{proof}
Statement~\ref{enum:lm:specpt:Y_is_h} is an immediate consequence of~\eqref{equ:hclA}.

\ref{enum:lm:specpt:A_is_h}
To prove that $B = (\Lsp\setminus\specPoints)\cup \{z\}$ is Hausdorff it suffices to verify that $\hclA{y}{B}=\{y\}$ for all $y\in B$.

If $y\in \Lsp \setminus \specPoints$, then $\hcl{y} = \{y\}$, whence $B \cap \hcl{y} = \{y\}$.
Since $z \in \hcl{w}$ if and only if $w \in \hcl{z}$, it follows that $\hcl{z} \subset \specPoints$.
Therefore
\begin{align*}
	B \cap \hcl{z} &=
	\bigl((\Lsp\setminus\specPoints)\cup \{z\} \bigr) \cap \hcl{z} =\\
	&=
	\bigl( (\Lsp\setminus\specPoints) \cap \hcl{z} \bigr)
	\cup
	\bigl(  \{z\} \cap \hcl{z}\bigr)=
	\varnothing \cup \{z\} = \{z\}.
\end{align*}

\ref{enum:lm:specpt:comp}
Let $y \notin \Lsp   \cup \specPoints$.
We will show that then there exists an open neighborhood $W$ of $y$ such that $\Lsp \cap W = \varnothing$.
This will imply that $y\not\in\Cl{\Lsp}$, whence $\Cl{\Lsp} \subset \Lsp \cup \specPoints$.

Since $y$ is not a \branch\ point, \myemph{i.e.}, $\hcl{y} = \{y\}$, we get from~\eqref{equ:hcl} that
\[
\Lsp \ \subset \ \Lsp \cup \specPoints \ \subset \ \Ysp \setminus \{y\} \ = \ \bigcup_{V \ni y} (\Ysp \setminus \Cl{V}),
\]
where $V$ runs over all open neighborhoods of $y$.
Thus
\[
\{
\Ysp \setminus \Cl{V} \mid V \ \text{is a neighborhood of $y$}
\}
\]
is an open cover of a compact set $\Lsp$, and so it contains a finite subcover, \myemph{i.e.}, one can find open neighborhoods $V_1,\ldots,V_m$ of $y$ such that
\[
\Lsp \subset \bigcup_{i = 1}^m (\Ysp \setminus \Cl{V_i}).
\]
Hence $W = V_1 \cap \ldots \cap V_m$ is an open neighborhood of $y$ with $\Lsp \cap W = \emptyset$.
\end{proof}

\begin{remark}\rm
For a Hausdorff space $\Ysp$, statement~\ref{enum:lm:specpt:comp} of Lemma~\ref{lm:specpt} is well known and claims that every compact subset $\Lsp$ of $\Ysp$ is closed, \emph{i.e.}, $\Cl{\Lsp}=\Lsp$.
\end{remark}

\section{Locally finite subsets}\label{sect:loc_fin_subsets}
Say that a subset $\Asp \subset \Ysp$ is \myemph{locally finite} if for every point $z\in\Ysp$ there exists an open neighborhood $\Usp$ such that the intersection $\Usp\cap \Asp$ is finite.
In other words, the family $\{ \{a\} \mid a\in \Asp\}$ of one-point subsets of $\Asp$ is locally finite in $\Ysp$.




\begin{lemma}\label{lm:T1_locfin}
Consider the following conditions on a subset $\Asp\subset\Ysp$ of a topological space $\Ysp$:

\begin{enumerate}[label={\rm(\arabic*)}, itemsep=0.5ex]
\item\label{enum:lm:T1_locfin:closed_discr}
$\Asp$ is closed and discrete.

\item\label{enum:lm:T1_locfin:U}
for each $y\in\Ysp$ there exists an open neighborhood $\Vsp$ intersecting $\Asp$ in at most one point;

\item\label{enum:lm:T1_locfin:SP_discrete}
$\Asp$ is locally finite;
\end{enumerate}
Then we have the following implications:
\ref{enum:lm:T1_locfin:closed_discr}
$\Rightarrow$
\ref{enum:lm:T1_locfin:U}
$\Rightarrow$
\ref{enum:lm:T1_locfin:SP_discrete}.

If $\Ysp$ is $T_1$ then we also have that~\ref{enum:lm:T1_locfin:SP_discrete}$\Rightarrow$\ref{enum:lm:T1_locfin:closed_discr}, i.e., all the above conditions are equivalent.
\end{lemma}
\begin{proof}
\ref{enum:lm:T1_locfin:closed_discr}$\Rightarrow$\ref{enum:lm:T1_locfin:U}.
Suppose $\Asp$ is closed and discrete and let $y\in\Ysp$.
If $y\not\in\Asp$, then $\Vsp = \Ysp\setminus\Asp$ is an open neighborhood of $y$ that does not intersect $\Asp$.
If $y\in\Asp$, then discreteness of $\Asp$ implies that there exists an open neighborhood $\Vsp$ of $y$ such that $\Vsp\cap\Asp = \{y\}$.

The implication \ref{enum:lm:T1_locfin:U}$\Rightarrow$\ref{enum:lm:T1_locfin:SP_discrete} is evident.

It remains to prove the implication \ref{enum:lm:T1_locfin:SP_discrete}$\Rightarrow$\ref{enum:lm:T1_locfin:closed_discr} under the assumption that $\Ysp$ is $T_1$.
Suppose $\Asp$ is locally finite.
Then each subset $B\subset\Asp$ is locally finite as well.
Moreover, as every point $y\in\Ysp$ is a closed subset, it follows that $B$ is closed as a union of a locally finite family of its closed one-point subsets.
In other words every subset of $\Asp$ is closed in $\Ysp$.
Hence $\Asp$ is closed and discrete.
\end{proof}

The following statement is a principal result of the present paper.

\begin{theorem}\label{th:fK_L}
Let $\Xsp$ be a locally compact Hausdorff topological space, $\Ysp$ a $T_1$-space whose set $\specPoints$ of \branch\ points is locally finite, $\prj:\Xsp\to\Ysp$ an open continuous and surjective map.
Then $\prj$ has properties \COMP and \CONT.
\end{theorem}
\begin{proof}
Due to Lemma~\ref{lm:cond_for_cont_ahom} it suffices to verify only \COMP.

Since $\specPoints$ is locally finite, it follows from Lemma~\ref{lm:T1_locfin} that each point $y\in\Ysp$ has an open neighborhood $\Vsp_{y}$ intersecting $\specPoints$ in at most one point, that is
\[
\specPoints  \cap \Vsp_{y} \ \subset \ \{y\}.
\]

As $\Xsp$ is locally compact Hausdorff, for each $x\in \prj^{-1}(\Lsp)$ there exists an open neighborhood $\Usp_{x}$ with compact closure $\Cl{\Usp_x}$ such that $\Cl{\Usp_x} \subset \prj^{-1}(\Vsp_{\prj(x)})$.
Then the following lemma holds:
\begin{sublemma}\label{lm:sub:fK_L}
We have that
\begin{equation}\label{equ:L_pclUx__clL_pclUx}
\Lsp \,\cap\, \prj(\Cl{\Usp_x}) \,=\, \Cl{\Lsp} \,\cap\, \prj(\Cl{\Usp_x}).
\end{equation}
This implies that
\begin{enumerate}[label={\rm(\alph*)}, itemsep=0.8ex]
\item\label{enum:lm:fK_L:1}
$\prj^{-1}(\Lsp \cap \prj(\Cl{\Usp_x}))$ is closed in $\prj^{-1}\bigl(\prj(\Cl{\Usp_x})\bigr)$;
\item\label{enum:lm:fK_L:2}
the set $\Ksp_x := \prj^{-1}(\Lsp \cap \prj(\Cl{\Usp_x})) \, \cap \, \Cl{\Usp_{x}}$ is compact;
\item\label{enum:lm:fK_L:3} $\prj(\Ksp_x) = \Lsp \cap \prj(\Cl{\Usp_x})$.
\end{enumerate}
\end{sublemma}
\begin{proof}
\eqref{equ:L_pclUx__clL_pclUx}
The inclusion $\Lsp \cap \prj(\Cl{\Usp_x}) \subset \Cl{\Lsp} \cap \prj(\Cl{\Usp_x})$ is trivial.
Conversely, by~\ref{enum:lm:specpt:comp} of Lemma~\ref{lm:specpt} $\Cl{\Lsp} \subset \Lsp \cup \specPoints$, whence
\[
	\Cl{\Lsp} \,\cap\, \prj(\Cl{\Usp_x}) \ \subset \
	(\Lsp \cup \specPoints) \cap \prj(\Cl{\Usp_x}) \ = \
	\bigl(\Lsp \cap \prj(\Cl{\Usp_x})\bigr) \cup
	\bigl(\specPoints \cap \prj(\Cl{\Usp_x})\bigr).
\]
But
\[
\specPoints \cap \prj(\Cl{\Usp_x}) \ \subset \
\specPoints \cap \Vsp_{\prj(x)} \ \subset \
\{\prj(x)\} \ \subset \
\Lsp \cap \prj(\Cl{\Usp_x}),
\]
whence $\Cl{\Lsp}\, \cap\, \prj(\Cl{\Usp_x}) \subset \Lsp \cap \prj(\Cl{\Usp_x})$ as well.

\medskip
\ref{enum:lm:fK_L:1}
Due to~\eqref{equ:L_pclUx__clL_pclUx}, the intersection $\Lsp \cap \prj(\Cl{\Usp_x})=\Cl{\Lsp} \cap \prj(\Cl{\Usp_x})$ is closed in $\prj(\Cl{\Usp_x})$.
Therefore $\prj^{-1}(\Lsp \cap \prj(\Cl{\Usp_x}))$ is closed in $\prj^{-1}\bigl(\prj(\Cl{\Usp_x})\bigr)$ by continuity of $\prj$.

\medskip

\ref{enum:lm:fK_L:2}
By~\ref{enum:lm:fK_L:1} $\Ksp_x := \prj^{-1}(\Lsp \cap \prj(\Cl{\Usp_x}))  \cap  \Cl{\Usp_{x}}$ is an intersection of two closed subsets of $\prj^{-1}(\prj(\Cl{\Usp_x}))$.
Hence $\Ksp_x$ is a closed subset of compact set $\Cl{\Usp_{x}}$, and therefore it is compact as well.

\medskip

\ref{enum:lm:fK_L:3}
Proof of this statement is based on the following simple observation.

\emph{Let $\prj: U \to V$ be a map between sets $U$ and $V$.
Then for any subsets $A \subset U$ and $B \subset V$ we have that $\prj(\prj^{-1}(B) \cap A) = B \cap \prj(A)$.}

In particular, for $A = \Cl{\Usp_x}$ and $B = \Lsp \cap \prj(\Cl{\Usp_x})$ we obtain that
\[
\prj(\Ksp_x) =
\prj \bigl( \prj^{-1}(\Lsp \cap \prj(\Cl{\Usp_x}))  \cap  \Cl{\Usp_{x}} \bigr) =
(\Lsp \cap \prj(\Cl{\Usp_x})) \cap \prj(\Cl{\Usp_x}) = \Lsp \cap \prj(\Cl{\Usp_x}).
\]
This completes Lemma~\ref{lm:sub:fK_L}.
\end{proof}

Now we can deduce Theorem~\ref{th:fK_L} from statement~\ref{enum:lm:fK_L:3} of Lemma~\ref{lm:sub:fK_L}.
Since $\prj$ is an open map, $\prj(\Usp_{x})$ is an open neighborhood of $\prj(x)$, whence by compactness of $\Lsp$ one can find finitely many points $x_1,\ldots,x_n\in \prj^{-1}(\Lsp)$ such that
\[ \Lsp \ \subset \ \mathop{\cup}\limits_{i=1}^{n} \prj(\Usp_{x_i})
\ \subset \ \mathop{\cup}\limits_{i=1}^{n} \prj(\Cl{\Usp_{x_i}}).
\]
Put $\Ksp = \mathop{\cup}\limits_{i=1}^{n} \Ksp_{x_i}$.
Then $\Ksp$ is compact and
\[
\prj(\Ksp) =
\prj\Bigl( \mathop{\cup}\limits_{i=1}^{n} \Ksp_{x_i} \Bigr)  =
\mathop{\cup}\limits_{i=1}^{n} \prj(\Ksp_{x_i})
\stackrel{\ref{enum:lm:fK_L:3}}{=}
\mathop{\cup}\limits_{i=1}^{n} \bigl( \Lsp \cap \prj(\Cl{\Usp_{x_i}}) \bigr) =
\Lsp.
\]
Theorem~\ref{th:fK_L} is proved.
\end{proof}

\subsection{Proof of Theorem~\ref{th:fK_L__foliations}}\label{sect:ect:proof:th:fK_L__foliations}
Let $\Partition$ be a foliation on a topological manifold $\Xsp$ such that the space of leaves $\Ysp$ is $T_1$ and the set $\specPoints$ of \branch\ points of $\Ysp$ is locally finite.

Then $\Xsp$ is a locally compact Hausdorff topological space, and the map $\prj:\Xsp\to\Ysp$ onto the spaces of leaves is open.
Hence by Theorem~\ref{th:fK_L} the map~\eqref{equ:ahom_EXP_EY}, $\ahom:\HXP\to\HY$, is continuous with respect to compact open topologies.
\qed

\section{Examples}\label{sect:examples}
In fact conditions \CONT and \COMP hold in many situations not covered by Corollary~\ref{cor:cond_for_cont_ahom} and Theorem~\ref{th:fK_L}.
We will consider several examples in which $\Ysp$ has more ``pathological'' properties.
In spite of certain triviality of statements below, they describe situations which very often appear in the foliations theory.

Thus again we will assume that $\Partition$ is a partition of a topological space $\Xsp$, $\Ysp$ is the corresponding space of leaves, $\prj:\Xsp\to\Ysp$ is the natural projection, and we endow $\Ysp$ with the factor topology with respect to $\prj$.

In order to check continuity of the induced homomorphism $\ahom:\HXP\to\HY$ and to verify property \COMP for the projection map $\prj$ we will use two following lemmas.

\subsection*{Trivial topology on $\Ysp$}
Recall that the topology on a set $\Ysp$ consisting only of two sets $\{\varnothing, \Ysp\}$ is called \myemph{trivial} (or \myemph{antidiscrete}).

\begin{lemma}\label{lm:each_orbit_dense}
Suppose either of the following conditions holds:
\begin{enumerate}[label={\rm(\alph*)}]
\item\label{enum:Y_triv:leaf_dense}
each element of $\Partition$ is everywhere dense;
\item\label{enum:Y_triv:top_triv}
$\Ysp$ has trivial topology.
\end{enumerate}
Then the compact open topology on $\CYY$ is also trivial, whence any map into $\CYY$ is continuous.
In particular, $\ahom$ is so, that is $\prj$ has property \CONT.
\end{lemma}

\begin{proof}
One easily checks that conditions~\ref{enum:Y_triv:leaf_dense} and~\ref{enum:Y_triv:top_triv} are equivalent.

Now, since the topology of $\Ysp$ contains only finitely many non-empty open sets (in fact a unique such set), it follows that every subset $\Lsp$ of $\Ysp$ is compact.
Hence every prebase set of compact open topology of $\CYY$ has the form
\[
\NKU{\Lsp}{\Ysp} = \{ f\in\CYY \mid f(\Lsp) \subset \Ysp\} = \CYY,
\]
where $\Lsp$ is an arbitrary non-empty subset of $\Ysp$.
In other words, the compact open topology on $\CYY$ contains only one non-empty set $\CYY$, and therefore is trivial.
\end{proof}

\subsection*{Baire spaces}
Let $E$ be a subset of a topological space $X$.
Then $E$ is \emph{of first category in $X$} if $E$ can be presented as a countable union of subsets which are nowhere dense in $X$.
Otherwise, $E$ is said to be \emph{of second category in $X$}, i.e. it can not be presented as a countable union of nowhere dense subsets of $X$.
\emph{A Baire space} is a topological space $X$ which is of second category in itself.

\begin{lemma}\label{lm:Gact_BaireSpaces}
Let $G$ be a topological group acting on a topological space $\Xsp$, $\Partition$ the partition of $X$ into $G$-orbits, $\Ysp = \Xsp / G$ the quotient space, and $\prj: \Xsp \to \Ysp$ the natural projection.
Suppose that
\begin{enumerate}[label={\rm(\alph*)}]
\item $\Xsp$ is a Hausdorff Baire space;
\item $G = \mathop{\cup}\limits_{i=1}^{\infty} G_i$ can be represented as a union of countably many compact sets $G_i$;
\item\label{enum:lm:Gact_BaireSpaces:2orb} $G$-action has at least two distinct orbits;
\item\label{enum:lm:Gact_BaireSpaces:dense_orb} each orbit of $G$-action is dense in $X$;
\end{enumerate}
Then $\prj$ has property \CONT but not \COMP.
\end{lemma}

\begin{proof}
Condition \CONT follows from~\ref{enum:lm:Gact_BaireSpaces:dense_orb} and Lemma~\ref{lm:each_orbit_dense}.
That lemma also implies that $\Ysp$ has trivial topology, and therefore every subset $\Lsp\subset\Ysp$ is compact.

By assumption~\ref{enum:lm:Gact_BaireSpaces:2orb}, $\Ysp$ contains at least two distinct points, whence there is a proper subset $\Lsp$ of $\Ysp$.
Then its complement $\Lsp' = \Ysp \setminus \Lsp$ is also a proper subset of $\Ysp$.
Thus $\Lsp$ and $\Lsp'$ are proper compact subsets of $\Ysp$.
Denote
\begin{align*}
A&:= \prj^{-1}(\Lsp), &
A'&:= \prj^{-1}(\Lsp') = \Xsp \setminus A.
\end{align*}
Suppose there exist compact sets $\Ksp\subset A$ and $\Ksp' \subset A'$ such that $\prj(\Ksp) = \Lsp$ and $\prj(\Ksp') = \Lsp'$.
We will show that this contradicts to assumptions of the lemma, and thus will imply that condition \COMP fails for $\prj$.

Let $\mu:G \times\Xsp\to\Xsp$ be the action map.
Then the relations $\prj(\Ksp) = \Lsp$ and $\prj(\Ksp') = \Lsp'$ mean that
\begin{align}\label{equ:X_1_cat}
A &= \mu(G \times \Ksp) = \bigcup_{i \in \bN} \mu(G_i \times \Ksp), &
A' &= \mu(G \times \Ksp') = \bigcup_{i \in \bN} \mu(G_i \times \Ksp').
\end{align}

Since every orbit of $\mu$ is dense in $\Xsp$, it follows that $A$ and $A'$ are dense in $\Xsp$ as well, whence their \myemph{compact} (and therefore closed in $\Xsp$) subsets $\mu(G_i \times \Ksp)$ and $\mu(G_i \times \Ksp')$ are nowhere dense in $\Xsp$.
Therefore~\eqref{equ:X_1_cat} implies that $\Xsp = A \sqcup A'$ is of first category, which contradicts to the assumption that $\Xsp$ is Baire.
Hence condition \COMP fails for $\prj$.
\end{proof}

\subsection*{Irrational rotation of the circle}
Let $S^1 =\{|z|=1\} \subset\bC$ be the unit circle in the complex plane, $\alpha \in (0, 1)$, and $f: S^1 \to S^1$ the rotation by the angle $2\pi\alpha$ given by $f(z)=ze^{2\pi i\alpha}$.
Then iterations $f^k$, $k \in \bZ$, of $f$ generate the action of $\bZ$ on $S^1$ given by the action map $\mu:\bZ \times S^1 \to S^1$, $ \mu(k, z) = f^k(z)$, $(k, z) \in \bZ \times S^1$.

Let $\Partition$ be the partition of $S^1$ into the orbits of this action, $\Ysp = S^1 / \Partition$ the quotient space, and $\prj: S^1 \to \Ysp$ the natural projection.

\begin{lemma}\label{lm:irr_rotation}
The map $\prj$ has property \CONT.
On the other $\prj$ has property \COMP iff $\alpha$ is rational.
\end{lemma}
\begin{proof}
If $\alpha$ is rational, then $\Ysp$ is homeomorphic with the circle and the projection map $\prj:S^1\to \Ysp$ is a locally trivial fibration (in fact a finite covering map).
In this case, due to Lemma~\ref{cor:cond_for_cont_ahom}\ref{enum:condstar:loc_triv_fibr}, the homomorphism $\ahom$ is continuous, i.e. $\prj$ has property \CONT.
Moreover, since $S^1$ is compact, the projection map is also proper, whence property \COMP holds as well.

On the other hand, if $\alpha$ is irrational, then the action of $\bZ$ on $S^1$ satisfies assumptions of Lemma~\ref{lm:Gact_BaireSpaces}, whence property \CONT holds, while property \COMP fails.
\end{proof}

\subsection*{Irrational flow of the torus}
More generally, let $T^2 = S^1\times S^1 \cong \bR^2 / \bZ^2$ be a $2$-torus, $\alpha\in\bR$, and $H:T^2\times\bR\to T^2$ be the flow defined by $H(x,y,t) = (x+t, y+\alpha t)$.
Let $\Partition$ be the foliation of $T^2$ by the orbits of $H$ and $\Ysp$ be the space of leaves.
\begin{lemma}\label{lm:irr_flow}
The map $\prj$ has property \CONT.
On the other $\prj$ has property \COMP iff $\alpha$ is rational.
\end{lemma}
The proof is literally the same as in Lemma~\ref{lm:irr_rotation}.

\subsection*{Denjoy example}
There is a well known class of orientation preserving circle homeomorphisms built by Arnauld Denjoy representatives of which have wandering intervals and irratonal rotation numbers, see~\cite{BrinStuck, KatokHasselblatt}.

Let $f : S^1 \to S^1$ be a homeomorphism from this class.
Then it is known to comply with the following two properties:
\begin{enumerate}[label=(\roman*), ref=(\roman*), itemsep=1ex]
\item\label{enum.Denjoy_1}
there exists a nowhere dense Cantor set $\Gamma \subset S^1$ such that $f(\Gamma) = \Gamma$ and for every $x\in\Gamma$ its orbit $\cO_f(x) = \mathop{\cup}\limits_{k \in \bZ} f^k(x)$ is dense in $\Gamma$;
\item\label{enum.Denjoy_2}
there exists an open arc $J_0 \subset S^1$ such that $S^1 \setminus \Gamma = \mathop{\sqcup}\limits_{m \in \bZ} f^m(J_0)$.
\end{enumerate}
Let $\Partition$ be the partition of $S^1$ by orbits of $f$, $\Ysp = S^1 / \Partition$ the quotient space, and $\prj: S^1 \to \Ysp$ the natural projection.

Denote $J_m = f^{m}(J_0)$, $m \in \bZ$.
Notice that $J_{m+1} = f(J_m)$ and $f|_{J_m}: J_m \to J_{m+1}$ is a homeomorphism for each $m$.

Let also $\tG = p(\Gamma)$ and $J = p(S^1\setminus\Gamma)$.
Since $\Gamma$ and $S^1\setminus\Gamma$ are saturated with respect to $\Partition$, we have that $\Ysp = \tG \sqcup J$.
Moreover, $\tG$ is closed and $J$ is open in $\Ysp$.
It is straightforward to check that $J$ is homeomorphic to $(0, 1)$ in the topology induced from $\Ysp$ and that each map $\prj|_{J_m}: J_m \to J$, $m \in \bZ$, is a homeomorphism.

\begin{figure}
\includegraphics[width=7cm]{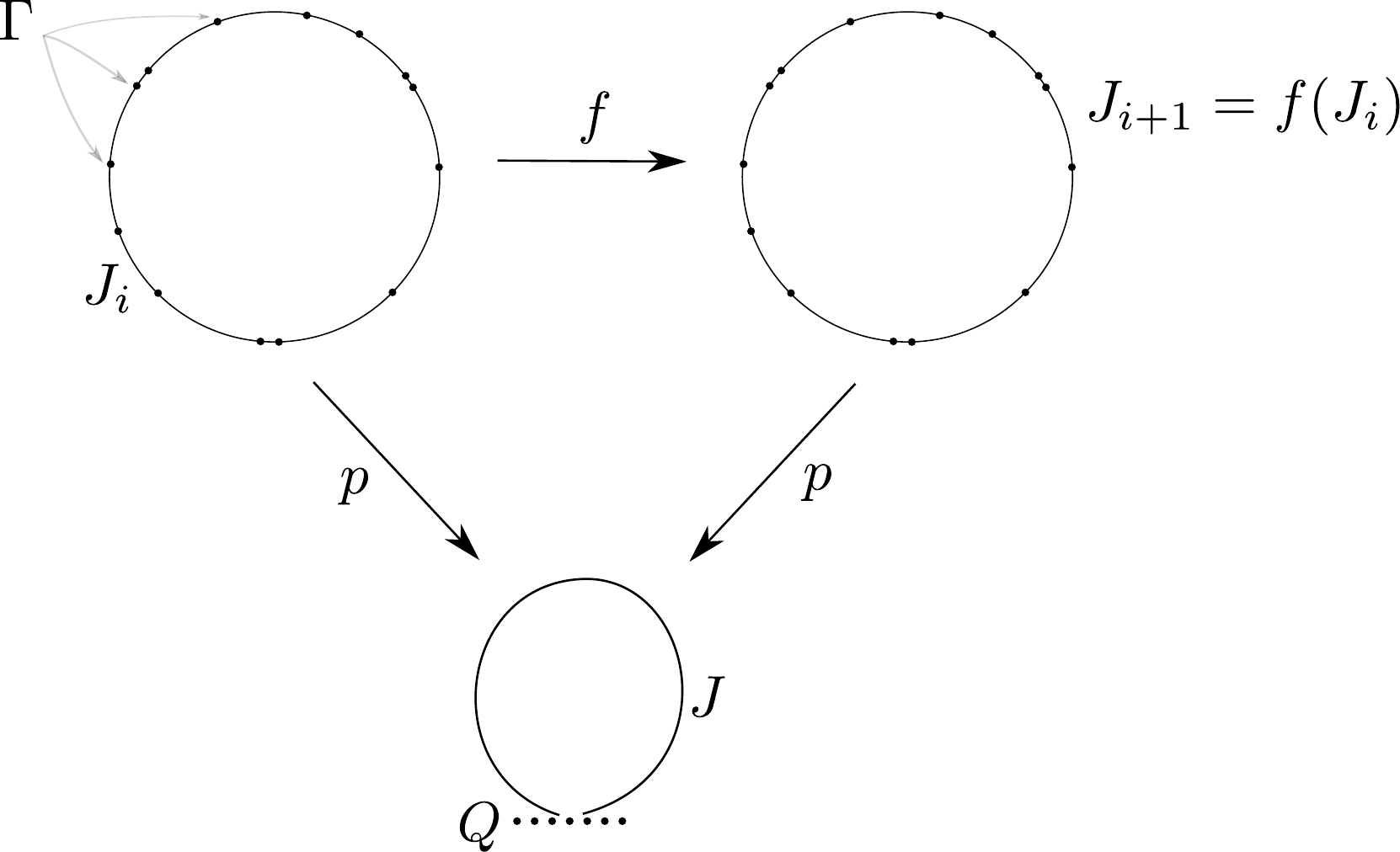}
\caption{}\label{fig:denjoy}
\end{figure}

Let us mention several properties of the topology of $\Ysp$.
\begin{lemma}\label{lm:denjoy_top_prop}
\begin{enumerate}[label={\rm(\alph*)}]
\item\label{enum:denj:sat_U}
Let $\Usp\subset S^1$ be an open subset.
If $\Usp\cap \Gamma\not=\varnothing$, then $\Usp$ intersects all orbits of $f$, so its saturation is $\prj^{-1}(\prj(\Usp))=S^1$.

\item\label{enum:denj:open_in_Y}
Every open subset $\Vsp \subset \Ysp$ intersecting $\tG$ coincides with $\Ysp$;

\item\label{enum:denj:comp_in_Y}
Every subset $\Lsp\subset \Ysp$ intersecting $\tG$ is compact.

\item\label{enum:denj:comt_to_tG}
Let $k:\Ysp\to\Ysp$ be a continuous map such that $k(y)\in J$ for some $y\in\tG$.
Then $k$ is a constant map into point $k(y)$.
\end{enumerate}
\end{lemma}
\begin{proof}
\ref{enum:denj:sat_U}
Let $y \in \Usp\cap \Gamma$ and $x\in S^1$ be any point.
If $x\in\Gamma$ as well, then by~\ref{enum.Denjoy_1} its orbit $\cO_f(x)$ is dense in $\Gamma$, whence the neighborhood $\Usp\cap\Gamma$ of $y$ in $\Gamma$ intersects $\cO_f(x)$.
In particular, $\Usp\cap\cO_f(x)\not=\varnothing$.

Otherwise, $x\in J_i$ for some $i\in\bZ$.
Since $\Cl{J_i} \cap \Gamma \neq \emptyset$ and diameters of the intervals $J_j$ tend to $0$ when $j\to\infty$, it follows that $J_m\subset\Usp$ for some $m\in\bZ$, whence $f^{m-i}(x) \in J_m\subset\Usp$.
Thus, $\Usp\cap\cO_f(x)\not=\varnothing$ as well.

\ref{enum:denj:open_in_Y}
Let $\Vsp\subset\Ysp$ be open subset intersecting $\tG$.
Then the open saturated set $\Usp:=\prj^{-1}(\Vsp)$ intersects $\Gamma$.
Hence $\Usp=S^1$ by~\ref{enum:denj:sat_U} and thus $\Vsp=\Ysp$.

\ref{enum:denj:comp_in_Y}
Let $y\in \Lsp \cap \tG$, $\xi = \{\Wsp_{j}\}_{j\in\Lambda}$ be an open cover of $\Lsp$, and $\Wsp_i$ be any element containing $y$.
Then due to~\ref{enum:denj:open_in_Y}, $\Wsp_i=\Ysp$, whence $\{\Wsp_i\}$ is a one-element subcover of $\xi$ which covers $\Lsp$.

\ref{enum:denj:comt_to_tG}
Let $\Vsp\subset J$ be a neighborhood of $k(y)$.
Then $k^{-1}(\Vsp)$ is an open neighborhood of $y\in\tG$, whence by~\ref{enum:denj:sat_U}, $k^{-1}(\Vsp) = \Ysp$.
Thus the image of $\Ysp$ is contained in arbitrary neighborhood of $k(y)$.
Since $J$ is a $T_1$-space, we have that
$\{ \kdif(y) \} = \mathop{\cap}\limits_{\substack{\kdif(y) \in V, \\ \text{$V$ is open}}} V$,
whence $k(\Ysp)\subset\{ \kdif(y) \}$, and therefore $k$ is a constant map.
\end{proof}

\newcommand\tV{\tilde{V}}
\newcommand\tU{\tilde{U}}

\begin{lemma}\label{lm:Denjoy_C_and_not_K}
For Denjoy homeomorphism $f:S^1\to S^1$ the quotient map $\prj:S^1\to\Ysp$ has property \CONT but does not have property \COMP.
\end{lemma}

\begin{proof}
{\em Verification of property \CONT.}
Let $\dif\in\End(S^1, \Partition)$, $k=\ahom(\dif) \in \EY = C(\Ysp, \Ysp)$, and
\[\NKU{\Lsp}{\Vsp} = \{ l \in \EY \mid l(\Lsp) \subset \Vsp \}\]
be a prebase set of the compact open topology on $\EY$ containing $k$, where $\Lsp \subset \Ysp$ is compact and $\Vsp \subset \Ysp$ is open.
In particular, $k(\Lsp) \subset \Vsp$.

Consider two cases.

a) If $\Vsp\cap\tG\not=\varnothing$, then $\Vsp=\Ysp$ by Lemma~\ref{lm:denjoy_top_prop}\ref{enum:denj:open_in_Y}.
Hence $\NKU{\Lsp}{\Vsp} = \NKU{\Lsp}{\Ysp}=\EY$.
Therefore $\ahom^{-1}\bigl(\NKU{\Lsp}{\Vsp}\bigr)=\ahom^{-1}\bigl(\EY\bigr)=\End(S^1, \Partition)$ is open.

b) Otherwise, $\Vsp\cap\tG=\varnothing$, i.e. $\Vsp\subset J$.
Again consider two subcases.

b1) Suppose $\Lsp\cap\tG\not=\varnothing$.
Since $k(\Lsp) \subset \Vsp$, we get from Lemma~\ref{lm:denjoy_top_prop}\ref{enum:denj:comt_to_tG} that $k$ is a constant map into some point $y\in\Ysp$.
As $\prj\circ\dif=k\circ \prj$, it follows that $\dif(S^1) \subset \prj^{-1}(y)$.
But $p^{-1}(y)$ has no nontrivial connected subsets, whence $\dif$ is a constant map, and its image is contained in some interval $J_i$.
Let $\Usp_i = J_i \cap \prj^{-1}(\Vsp)$.
Then $\mathcal{U} = \NKU{S^1}{\Usp_i} \cap \End(S^1,\Partition)$ is a neighborhood of $\dif$ in $\End(S^1,\Partition)$ and $\ahom\bigl(\mathcal{U}\bigr) \subset \NKU{\Lsp}{\Vsp}$.

b2) Finally, assume that $\Lsp\cap\tG=\varnothing$, so $\Lsp,\Vsp\subset J$.
Denote $\Ksp_0 = \prj^{-1}(\Lsp)\cap J_0$ and $\Usp = \prj^{-1}(\Vsp)$.
Then the restriction $\prj|_{\Ksp_0}: \Ksp_0 \to \Lsp$ is a homeomorphism, whence $\Ksp_0$ is compact.
Moreover, as $\prj\circ\dif=k\circ \prj$, we have that $\dif(\Ksp_0) \subset \Usp$, whence $\mathcal{U} := \NKU{\Ksp_0}{\Usp} \cap \End(S^1,\Partition)$ is a neighborhood of $\dif$ in $\End(S^1,\Partition)$.

We claim that $\ahom\bigl(\mathcal{U}\bigr) \subset \NKU{\Lsp}{\Vsp}$.
Indeed, if $g\in \mathcal{U}$, so $g(\Ksp_0)\subset \Usp$, then
\[
\ahom(g)(\Lsp) =
\ahom(g)\circ p(\Ksp_0) =
p \circ g(\Ksp_0) \subset p(\Usp) \subset \Vsp.
\]

Thus $\ahom$ is continuous, i.e. property \CONT holds for $\prj$.

\medskip

{\em Proof that \COMP fails.}
Consider the restriction $f_0 = f|_{\Gamma} : \Gamma \to \Gamma$. Since $\Gamma$ is invariant under $f$ by property~\ref{enum.Denjoy_1} and $\tG = \prj(\Gamma)$, then $\prj_0 = \prj|_{\Gamma} : \Gamma \to \tG$ is well defined projection of $\Gamma$ onto the quotient space $\tG$.

It is known from Baire theorem that $\Gamma$ is a Baire space, so we can apply Lemma~\ref{lm:Gact_BaireSpaces} to the action of $\bZ$ on $\Gamma$ given by its action map $\mu:\bZ \times \Gamma \to \Gamma$, $ \mu(k, z) = f^k(z) = f_0^k(z)$, $(k, z) \in \bZ \times \Gamma$. So, there exists a compact subset $\Lsp$ of $\tG$ such that there is no compact subset of $\Gamma$ which projects onto $\Lsp$ under $\prj_0$.

The property of being compact is intrinsic, hence there is no compact subset of $S^1$ which projects onto $\Lsp$ under $\prj$.
\end{proof}

\begin{remark}
Let $H$ be either $\bZ$ or $\bR$.
Consider a (topological) dynamical system $H : G \times \Xsp \to \Xsp$ on a locally compact Hausdorff space $\Xsp$. Let $\Ysp$ be a space of orbits of this dynamical system endowed with quotient topology and $\prj : \Xsp \to \Ysp$ be a natural projection.

Each closed subset $A$ of $\Xsp$ is a Baire space in the topology induced from $\Xsp$ by Baire category theorem since $A$ is a locally compact Hausdorff space.
Hence the considerations similar to ones made in the second part of the proof of Lemma~\ref{lm:Denjoy_C_and_not_K} show the following.
\begin{itemize}
\item[] If there exists a minimal subset of $\Xsp$ which contains more than one orbit, then the projection $\prj$ does not comply with the property~\COMP.
\end{itemize}
\end{remark}

\bibliographystyle{proc_igc_plain}
\bibliography{actions_of_foliated_homeomorphisms}

\end{document}